\documentclass[12pt, reqno]{amsart}
\usepackage{fullpage}
\usepackage[english]{babel}
\usepackage[utf8]{inputenc}
\usepackage[svgnames]{xcolor}
\usepackage{hyperref}
\hypersetup{
    colorlinks=true,
    linkcolor={MediumBlue},
    citecolor={teal},
    urlcolor={teal}
}
\usepackage[nameinlink,capitalise]{cleveref}
\usepackage{amssymb}
\usepackage{enumitem}
\usepackage{centernot}
\usepackage[round]{natbib}
\usepackage{multirow}

\usepackage{setspace}

\newcommand{\R}{\mathbb{R}}

\newcommand{\N}{\mathbb{N}}

\newcommand{\ind}{\perp\!\!\!\perp}

\numberwithin{equation}{section}
\theoremstyle{plain}
\newtheorem{theorem}{Theorem}[section]

\theoremstyle{plain}
\newtheorem{proposition}[theorem]{Proposition}

\theoremstyle{remark}
\newtheorem{remark}[theorem]{Remark}

\theoremstyle{definition}
\newtheorem{example}[theorem]{Example}

\theoremstyle{plain}
\newtheorem{lemma}[theorem]{Lemma}

\theoremstyle{plain}

\theoremstyle{plain}

\title{On Weak Convergence of Gaussian Conditional Distributions}

\author{Sarah Lumpp}
\address{TUM School of Computation, Information and Technology, Technical University of Munich}
\email{sarah.lumpp@tum.de}

\author{Mathias Drton}
\address{TUM School of Computation, Information and Technology, Technical University of Munich, and Munich Center for Machine Learning, Germany}
\email{mathias.drton@tum.de}

\keywords{Conditional distribution, Schur complement, singular normal distribution, matrix determinant lemma, Lyapunov equation}

\begin{document}

\begin{abstract}
    Weak convergence of joint distributions generally does not imply convergence of conditional distributions. In particular, conditional distributions need not converge when joint Gaussian distributions converge to a singular Gaussian limit. Algebraically, this is due to the fact that at singular covariance matrices, Schur complements are not continuous functions of the matrix entries.  Our results lay out special conditions under which convergence of Gaussian conditional distributions nevertheless occurs, and we exemplify how this allows one to reason about conditional independence in a new class of graphical models.
\end{abstract}

\maketitle

\section{Introduction}

Let $X \sim \mathcal{N}(0,\Sigma)$ be a possibly singular Gaussian random vector taking values in $\mathbb{R}^p$; assuming zero mean is without loss of generality.  Let $[p]=\{1,\dots,p\}$, and let
$S,R\subseteq [p]$ be two disjoint subsets. Then the conditional covariance matrix of the subvector $X_R$ given $X_S$ is
\begin{equation}
    \label{eq:conditional-covariance}
    \Sigma_{R\mid S} := \Sigma_{R,R} - \Sigma_{R,S}\left(\Sigma_{S,S}\right)^+\Sigma_{S,R},
\end{equation}
where $(\cdot)^+$ denotes a pseudoinverse. Let $X^{(m)} \sim\mathcal{N}(0,\Sigma^{(m)})$ be a sequence of random vectors  that converges weakly to  $X$, i.e., the matrices $\Sigma^{(m)}$ converge to $\Sigma$. Our interest is in conditions ensuring convergence of the conditional distributions $\mathcal{N}(0,\Sigma^{(m)}_{R\mid S})$ to $\mathcal{N}(0,\Sigma_{R\mid S})$. This is trivially true when $\Sigma$ is positive definite, and our focus will be on singular matrices $\Sigma$.

Weak convergence of joint distributions need not imply convergence of conditional distributions; see, e.g., \cite{lauritzen2024}. This is reflected in the pseudoinverse in~\eqref{eq:conditional-covariance} not being a continuous function of $\Sigma$ \citep[Example 3.11]{lauritzen1996graphical}.  
In specific situations, however, conditional covariance matrices may still converge; one instance is sequences of matrices $\Sigma^{(m)}$ of constant rank \citep{stewart1969continuity, ben2006generalized}. 
Motivated by recent work on conditional independence \citep{boege2024conditional}, we consider a more challenging case where every matrix $\Sigma^{(m)}$ has full rank, while the limiting covariance matrix $\Sigma$ is  singular.

\medskip
\textbf{Notation.} 
For $\alpha\subseteq[p]$, let $s(\alpha)=\sum_{a\in \alpha} a$; note that $(-1)^{s([p])}=(-1)^{\lceil p/2 \rceil}$.
Further, let $\mathbf{1}_p = (1,\dots,1)^T \in \R^p$, and write $\mathrm{im}(\Sigma)$ for the column span of matrix $\Sigma$. The concatenation of two matrices $A$ and $B$ with the same number of rows is denoted as $[A|B]$. For integers $r < k$, we define the map
\begin{align*}
    f_{\mathrm{asy}}: \quad & \R^{k \times r} \times \R^{k \times r} \times \R^{k \times k} \to \R, \\
    & \left(U,V,G\right) \mapsto
    (-1)^{s([k-r])} \sum_{\substack{\alpha \subseteq [k] \\ |\alpha|=k-r}} 
    (-1)^{s(\alpha)} 
    \det\left(\left[\begin{array}{c|c}
        G_{[k],\alpha} & U \\
    \end{array}\right]\right) 
    \det\left(V_{[k]\setminus \alpha,[r]}\right),
\end{align*}
and its restriction
\begin{equation}
f(U,G) = f_{\mathrm{asy}}(U,U,G).
\end{equation}
Here, the original order of the selected rows and columns is preserved in the determinants. 

\begin{example} \label{ex:f}
    Take $G=(g_{i,j}) \in \R^{3 \times 3}$ and set 
    $
    U = \left(\begin{smallmatrix}
        1 & 0  \\
        0 & 1  \\
        0 & 1 \\
    \end{smallmatrix}\right).
    $
    Then
    \[{f(U,G) 
    = \det \begin{pmatrix}
        g_{2,2} & 1 \\ g_{3,2} & 1
    \end{pmatrix} 
    - \det \begin{pmatrix}
        g_{2,3} & 1 \\ g_{3,3} & 1
    \end{pmatrix}
    = g_{2,2} - g_{3,2} + g_{3,3} - g_{2,3}}.\] 
\end{example}

The following theorem constitutes our main result.

\begin{theorem} \label{thm:main_convergence}
    Let $\Sigma^{(m)}$, $m\in\mathbb{N}$, be a sequence of invertible symmetric $p \times p$ matrices with singular limit $\Sigma$ and admitting the expansion 
    \[
    \Sigma^{(m)} = \Sigma + \frac1m \Sigma^{(1,\infty)} + o \left(\frac{1}{m}\right).
    \]
    Let $S\subset[p]$ with $k =|S|< p$, and let $R=[p]\setminus S$.
    If the matrices $\Sigma$ and $\Sigma^{(1,\infty)}$ satisfy that 
    \begin{itemize}
        \item[(i)] $f \big(U,\Sigma^{(1,\infty)}_{S,S}\big) \neq 0$ for a matrix $U$ from a symmetric rank decomposition $\Sigma_{S,S} = UBU^T$ with $U \in \R^{k \times r}$ and $B \in \R^{r \times r}$ of full rank $r = \mathrm{rank}(\Sigma_{S,S})$, and
        \item[(ii)] $\mathrm{im}(\Sigma_{S,R})\subseteq \mathrm{im}(\Sigma_{S,S})$,
    \end{itemize}
    then the conditional covariance matrices $\Sigma^{(m)}_{R\mid S}$ converge to $\Sigma_{R\mid S}$. 
\end{theorem}

In the theorem, the conditional covariance matrices $\Sigma^{(m)}_{R\mid S} = \Sigma^{(m)}_{R,R} - \Sigma^{(m)}_{R,S}\big(\Sigma^{(m)}_{S,S}\big)^{-1}\Sigma^{(m)}_{S,R}$ are Schur complements, whereas the limit involves a pseudoinverse as in~\eqref{eq:conditional-covariance}.  

\begin{remark}
    Condition $(i)$ in \Cref{thm:main_convergence} does not depend on the chosen decomposition. If $\Sigma_{S,S} = \Tilde{U}\Tilde{B}\Tilde{U}^T$, where the columns of $\Tilde{U}$ form again a basis of $\mathrm{im}(\Sigma_{S,S})$, then $\Tilde{U} = UC$ for an invertible $r \times r$ matrix $C$. Consequently, for any $G \in \R^{k\times k}$,
    \begin{align*}
        f(\Tilde{U},G) 
        =\!\!\!\! \sum_{\substack{\alpha \subseteq [k] \\ |\alpha|=k-r}} 
        (-1)^{s(\alpha) + s([k-r])} 
        \det\left(
        \left[\begin{array}{@{}c|c@{}}
            G_{[k],\alpha} & U \\
        \end{array}\right]
        \begin{bmatrix}
            I_k & 0 \\ 0 & C
        \end{bmatrix}\right) 
        \det\left(U_{[k]\setminus \alpha,[r]} C\right)
        = \det(C)^2 f(U,G).  
    \end{align*}
    Note that in condition $(i)$, $B$ could always be taken to be diagonal by spectral decomposition of $\Sigma_{S,S}$. But the added flexibility we allow for can be useful when applying the theorem.
\end{remark}

\begin{remark}
    \Cref{thm:main_convergence} also holds without symmetry. We then require $\mathrm{im}((\Sigma_{R,S})^T) \subseteq \mathrm{im}((\Sigma_{S,S})^T)$, $\mathrm{im}(\Sigma_{S,R}) \subseteq \mathrm{im}(\Sigma_{S,S})$, and $f_{\mathrm{asy}} \big(U,V,\Sigma^{(1,\infty)}_{S,S}\big) \neq 0$ for any factorization $\Sigma_{S,S} = UBV^T$ with $B \in \R^{r \times r}$ invertible and $U,V \in \R^{k \times r}$ of full column rank.
\end{remark}

In the sequel, \Cref{sec:application} briefly motivates the studied problem. \Cref{sec:proof} develops the proof of \Cref{thm:main_convergence}, for which we derive a generalization of the matrix determinant lemma.  We note that our work generalizes results from~\citet[Section 7]{berczi2023completecollineationsmaximumlikelihood}.  
Finally, \Cref{sec:toeplitz} studies our conditions for special Toeplitz matrices, giving an intuitive alternative proof of a key result in \cite{boege2024conditional}.  Some further details are deferred to the appendix.


\section{Motivating application}
\label{sec:application}

The issue of weak convergence of conditional distributions arises in studies of conditional independence in graphical models.  A recently introduced class of models  considers Gaussian distributions whose covariance matrices $\Sigma\in\mathbb{R}^{p\times p}$ satisfy the continuous Lyapunov equation
\begin{equation} \label{eq:lyap}
    M \Sigma + M \Sigma^T + 2I_p = 0
\end{equation}
for a stable matrix $M$ whose zero pattern encodes a graphical structure \citep{dettling2023,varando2020graphical}. 
\citet{boege2024conditional} characterize the conditional independence relations that hold in such distributions.
One of their key results asserts that certain paths known as treks create dependence between their endpoints that does not vanish when one conditions on all the random variables visited along the trek \citep[Prop.~4.5]{boege2024conditional}.
We exemplify the problem now and demonstrate a solution by way of \Cref{thm:main_convergence}.

\begin{example} \label{ex:application}
Let $p=4$, and consider the trek $\mathcal{T}_4: 1 \leftarrow 2 \rightarrow 3 \rightarrow 4$.  To prove Proposition 4.5 in \cite{boege2024conditional} for this trek, we must construct a stable matrix $M=(m_{ij})\in\mathbb{R}^{4\times 4}$ with $m_{ij}=0$ unless $i=j$ or $(i,j)\in\{(1,2),(3,2),(4,3)\}$ such that the solution $\Sigma$ of \eqref{eq:lyap} satisfies $\Sigma_{14|23}\not=0$; we use the shorthand $ij=\{i,j\}$.  This is easy, simply take
\[
M_{4} = \begin{pmatrix}
    -1 & 1 & 0 & 0  \\
    0 & -1 & 0 & 0  \\
    0 & 1 & -1 & 0  \\
    0 & 0 & 1 & -1
\end{pmatrix}, 
\quad\text{for which}\quad
\Sigma_{4}=
\begin{pmatrix}
 \tfrac{3}{2} & \tfrac{1}{2} & \tfrac{1}{2} & \tfrac{3}{8} \\[1.5mm]
 \tfrac{1}{2} & 1 & \tfrac{1}{2} & \tfrac{1}{4} \\[1.5mm]
 \tfrac{1}{2} & \tfrac{1}{2} & \tfrac{3}{2} & \tfrac{7}{8} \\[1.5mm]
 \tfrac{3}{8} & \tfrac{1}{4} & \tfrac{7}{8} & \tfrac{15}{8} 
\end{pmatrix}
\quad\text{and}\quad
(\Sigma_{4})_{14|23}=\frac{1}{10}.
\]
For general treks of arbitrary length, however, it becomes difficult to directly prove that an analogous construction leads to non-vanishing conditional covariance.
This is where \Cref{thm:main_convergence} becomes helpful and allows one to  turn the counterexample for the short trek into one for the long trek, via perfect correlation.  We illustrate this for $p=5$ and trek $\mathcal{T}_5: 1 \leftarrow 2 \rightarrow 3 \rightarrow 4 \rightarrow 5$, for which we consider the sequence of stable matrices
\[
M^{(m)}_5 = \begin{pmatrix}
    -1 & 1 & 0 & 0 & 0 \\
    0 & -1 & 0 & 0 & 0 \\
    0 & 1 & -1 & 0 & 0 \\
    0 & 0 & m & -m & 0 \\
    0 & 0 & 0 & 1 & -1
\end{pmatrix}.
\]
It is not hard to show that the associated solutions to \eqref{eq:lyap} satisfy 
$
\Sigma^{(m)}_5 = \Sigma_5 + \frac1m \Sigma^{(1,\infty)}_5 + o \left(\frac{1}{m}\right), 
$
where 
$
\Sigma_5 = {U}\Sigma_{4}{U}^T
$
with
\[
{U} = \begin{pmatrix}
    1 & 0 & 0 & 0 \\
    0 & 1 & 0 & 0 \\
    0 & 0 & 1 & 0 \\
    0 & 0 & 1 & 0 \\
    0 & 0 & 0 & 1 \\
\end{pmatrix}
\quad\text{and}\quad
(\Sigma^{(1,\infty)}_5)_{234,234} = \begin{pmatrix}
    0 & 0 & \frac12 \\
    0 & 0 & -1 \\
    \frac12 & -1 & 0
\end{pmatrix}.
\]
Due to perfect correlation between third
and fourth random variables, the singular limit $\Sigma_5$ has
$(\Sigma_5)_{15\mid 234} =(\Sigma_{4})_{1,4|2,3}=1/10$. It is easy to check that \Cref{thm:main_convergence} is applicable;  in particular, $(\Sigma_5)_{234,234} = U_{234,23}(\Sigma_4)_{23,23}U_{234,23}^T$ and $f(U_{234,23},\Sigma^{(1,\infty)}_{234,234}) = 2 \neq 0$, as can be deduced from \Cref{ex:f}.
\cref{thm:main_convergence} then gives the desired conclusion as for large enough $m$, the positive definite matrix $\Sigma^{(m)}_5$ must have $(\Sigma^{(m)}_5)_{15\mid 234}\not=0$.
\end{example}


\section{Proof of the theorem}
\label{sec:proof}

We start by proving two lemmas on asymptotics of determinants. The first lemma provides a generalized version of the matrix determinant lemma (compare \cref{sec:Appendix1}).

\begin{lemma} \label{lem:matrix_det_generalized}
  Let $U,V \in \R^{k\times r}$, $G \in \R^{k \times k}$, and $B \in
  \R^{r \times r}$ be matrices with dimensions satisfying $r<k$.  Then 
  \begin{align*}
    \det\left(UBV^T+ \frac{1}{m}G\right) = \ 
    &\frac{1}{m^{k-r}} \cdot \det(B) \cdot
      f_{\mathrm{asy}}(U,V,G) + o\left(\frac{1}{m^{k-r}}\right).
  \end{align*}
\end{lemma}

\begin{proof}
    Write 
    $
    UBV^T + \frac1m G 
    = \left[ \begin{array}{c|c}
     \frac1m G & U \\
    \end{array}\right] 
    \left[ \begin{array}{c}
     I_{k\times k} \\
     BV^T
    \end{array}\right] 
    $
    as a product of a $k \times (k+r)$ and a $(k+r) \times k$ matrix.
    The Cauchy-Binet formula yields that
    \begin{align}
      \label{eq:cauchy-binet}
        \det\left(UBV^T + \frac1m G\right) 
        &= \sum_{\substack{I \subseteq [k+r] \\ |I|=k}} 
        \det\left(\left[\begin{array}{c|c}
         \frac1m G & U \\
        \end{array}\right]_{[k],I}\right) 
        \det\left(\left[\begin{array}{c|c}
         I_{k\times k} & VB^T \\
        \end{array}\right]_{[k],I}\right). 
    \end{align}
    To track the row and column indices appearing in the
    formula,
    we partition an index set $I$ with $|I| = k$ 
    using the set $\alpha(I)=I\cap[k]$ and a shifted version of its complement, defined as $\beta(I)= \big\{i-k:i\in I\setminus[k]\big\} \subseteq [r]$.  Furthermore, we use the shorthand $\Bar{\alpha}(I) = [k]\setminus I$.
    
    Simplifying the first determinant in~\eqref{eq:cauchy-binet} yields
    \begin{align*}
        \det\left(\left[\begin{array}{c|c}
        \frac1m G & U \\
        \end{array}\right]_{[k],I}\right)
        = \det\left(\left[\begin{array}{c|c}
        \left(\frac1m G\right)_{[k],\alpha(I)} & U_{[k],\beta(I)} \\
        \end{array}\right]\right) 
        = \frac{1}{m^{|\alpha(I)|}} \det\left(\left[\begin{array}{c|c}
        G & U \\
        \end{array}\right]_{[k],I}\right),
    \end{align*}
    and expanding the second determinant along the columns of $\left(I_{k\times k}\right)_{[k],\alpha(I)}$
    yields
    \begin{align*} 
        \det\left(\left[\begin{array}{c|c}
        I_{k\times k} & VB^T \\
        \end{array}\right]_{[k],I}\right)
        = & \ \det\left(\left[\begin{array}{c|c}
        \left(I_{k\times k}\right)_{[k],\alpha(I)} & (VB^T)_{[k],\beta(I)} \\
        \end{array}\right]\right) \\
        = & \ (-1)^{\lceil |\alpha(I)| / 2\rceil}+s(\alpha(I))\det\left((VB^T)_{\Bar{\alpha}(I),\beta(I)}\right).
    \end{align*}
    Combining both results,
    we may write the determinant of $UBV^T + \frac1m G$ as
    \begin{align} 
      \label{eq:det_Xm}
      &\sum_{\substack{I \subseteq [k+r] \\ |I|=k}} 
        \frac{1}{m^{|\alpha(I)|}} 
        (-1)^{\lceil |\alpha(I)| / 2\rceil} + s(\alpha(I))
        \det\left(\left[\begin{array}{c|c}
            G_{[k],\alpha(I)} & U_{[k],\beta(I)} \\
        \end{array}\right]\right) 
        \det\left((VB^T)_{\Bar{\alpha}(I),\beta(I)}\right).
    \end{align}
    Note that $|\alpha(I)| = k - |\beta(I)|$ is minimal when $\beta(I) = [r]$. In these cases, $|\alpha(I)| = k - r$ independently of the chosen set $I$, so we can identify $\alpha(I)$ with sets $\alpha \subseteq [k]$ with $|\alpha| = k-r$. 
    For the index sets $I$ with $|\beta(I)| = r$,
    \[
    \det\left(\left[\begin{array}{c|c}
        G_{[k],\alpha(I)} & U_{[k],\beta(I)} \\
    \end{array}\right]\right) 
    = \det\left(\left[\begin{array}{c|c}
        G_{[k],\alpha} & U_{[k],[r]} \\
    \end{array}\right]\right)
    = \det\left(\left[\begin{array}{c|c}
        G_{[k],\alpha} & U \\
    \end{array}\right]\right)
    \]
    and 
    \[
    \det\left((VB^T)_{\Bar{\alpha}(I),\beta(I)}\right)
    = \det\left((VB^T)_{[k]\setminus \alpha,[r]}\right)
    = \det\left(V_{[k]\setminus \alpha,[r]}B^T\right)
    = \det\left(V_{[k]\setminus \alpha,[r]}\right) \det(B).
    \]
    Reindexing the sum in \eqref{eq:det_Xm} yields the claimed
    formula. 
\end{proof}

The next lemma extends the previous result to a setting with additional noise. Its proof is deferred to the appendix.

\begin{lemma} \label{lem:matrix_det_taylor}
    Consider the same setting as in \Cref{lem:matrix_det_generalized}.
    Additionally, let $D^{(m)}$ be a sequence of $k \times k$ matrices whose entries are of order $o(\frac{1}{m})$. Then 
    \begin{align*}
        \det\left(UBV^T+ \frac{1}{m}G+D^{(m)}\right) = \det\left(UBV^T+ \frac{1}{m}G\right) +  o\left(\frac{1}{m^{k-r}}\right).
    \end{align*}
\end{lemma}


Using the lemmas above, we prove the main result by employing Cramer's rule. 

\begin{proof}[Proof of \Cref{thm:main_convergence}]
    The limit of the conditional covariance matrix may be computed entry-by-entry. 
    Let $v_m,w_m\in\mathbb{R}^k$ be a choice of two not necessarily distinct columns of the submatrix $\Sigma^{(m)}_{S,R}$ with respective limits $v$ and $w$. For ease of readability, let 
    $
    A_m = \Sigma^{(m)}_{S,S}, \ A = \Sigma_{S,S}, \ G = \Sigma^{(1,\infty)}_{S,S} \in \R^{k \times k}
    $.
    To prove our claim, we will show $\lim_{m \to \infty} v_m^T A_m^{-1} w_m = v^T A^+ w$.
    
    \emph{Existence of the limit.}   
    Consider the vector 
    $
    z_m = A_m^{-1}w_m
    $
    with coordinates $z_{m,j}$, $j\in[k]$. To show convergence of each component $z_{m,j}$, we express it via Cramer's rule as
    \[
    z_{m,j} = \frac{\det\left(A_m[j]\right)}{\det\left(A_m\right)},
    \]
    where $A_m[j]$ is the matrix $A_m$ with the $j$-th column replaced by the vector $w_m$. 
    The denominator does not depend on $j$ and equals
    \begin{equation} \label{eq:denominator}
    \begin{aligned}
    \det\left(A_m\right) 
    &= \det\left(A + \frac1m G + o\left(\frac1{m}\right)\right)
    = \det\left(UBU^T + \frac1m G\right) + o\left(\frac1{m^{k-r}}\right) \\
    &= \frac{1}{m^{k-r}} \cdot \det(B) \cdot 
        f\left(U,G\right)+ o\left(\frac{1}{m^{k-r}}\right), 
    \end{aligned}
    \end{equation}
    where we used \Cref{lem:matrix_det_taylor} in the second equality and  \Cref{lem:matrix_det_generalized} in the third equality. Due to condition $(i)$ and $B$ being invertible,  $\det\left( A_m \right)$ is always of order $m^{-(k-r)}$.

    To show convergence of $z_{m,j}$, we have to prove that the numerator converges at least as fast as the denominator.
    We write $w_m = w + \frac1m h + o\left(\frac1{m}\right)$, so that  
    \[
    A_m[j] = A[j] +\frac1m G[j] + o\left(\frac1{m}\right),
    \]
    where $A[j]$ is the matrix $A$ with the $j$-th column replaced by $w$, and $G[j]$ is $G$ with the $j$-th column replaced by $h$. Factorize $A[j] = \Tilde{U}^T \Tilde{V}$ with $\Tilde{U}, \Tilde{V} \in \R^{k \times r}$. The same calculation as for the denominator yields 
    \begin{align} \label{eq:numerator}
      \det\left(A_m[j]
      \right)
        = \frac{1}{m^{k-r}} \cdot
        f_{\mathrm{asy}}(\Tilde{U},\Tilde{V},G[j]) + o\left(\frac{1}{m^{k-r}}\right). 
    \end{align}
    Note that in this step, it is not necessary for the coefficient of $\frac1{m^{k-r}}$ to be non-zero as we only need to verify that all coefficients of $\frac1{m^n}$ with $n < k-r$ are zero. 
    Thus, we conclude from \eqref{eq:denominator} and \eqref{eq:numerator} that the limit of $z_{m,j}$ exists for all $j \in [k]$.   

    \emph{Computing the limit.} Given that the limit $z = \lim_{m \to \infty}z_m$ exists, we have
    \begin{align*}
    A \cdot z 
    &= \lim_{m \to \infty} A_m \lim_{m \to \infty} z_m 
    = \lim_{m \to \infty} w_m 
    = w.
    \end{align*}
    Since $A=UBU^T$ with $U$ of full column rank and $B$ invertible, this can be rewritten to
    $
    U^T z = B^{-1} U^+ w,
    $
    and the Moore Penrose inverse of $A$ is 
    $
    A^+ = (U^+)^T B^{-1} U^+. 
    $
    Combining these results, we have 
    $
    A^+ w 
    = (U^+)^T B^{-1} U^+ w 
    = (U^+)^T U^T z 
    = (U U^+)^T z
    $
    and therefore
    $
    v^T A^+ w 
    = (U U^+ v)^T z$,    
    where $U U^+ v$ is an orthogonal projection of $v$ onto the column span of $U$. 
    Due to $(ii)$, $U U^+ v = v$ and, thus,
    \[
    \lim_{m \to \infty} v_m^T A_m^{-1} w_m
    = \lim_{m \to \infty} v_m^T \lim_{m \to \infty} z_m
    = v^T z 
    = v^T A^+ w. \qedhere
    \]
\end{proof}


\section{An application involving Toeplitz matrices} \label{sec:toeplitz}

For $n \in \N_{>0}$, consider the symmetric $n$-dimensional Toeplitz matrix
\begin{align*}
    T_n 
    = \left(- \lvert i - j \rvert \right)_{i,j=1,\dots,n}
    &= \begin{pmatrix}
        0 & -1 &  \dots & -(n-1) \\
        -1 & 0 & \dots & -(n-2) \\
        \vdots & \vdots & \ddots & \vdots \\
        -(n-1) & -(n-2) & \dots & 0 \\
    \end{pmatrix} \in \R^{n \times n}.
\end{align*}
Let $T_n[i]$ be the matrix where the $i$-th column of $T_n$ is replaced by the vector $\mathbf{1}_n \in \R^n$.
A direct calculation using facts about Toeplitz determinants yields the following fact.

\begin{lemma} \label{lem:det_GU}
    For $n_l,n_r\in \N_{>0}$, let $k = n_l + 1 + n_r$. Define
    \[
    U = \begin{pmatrix}
        \mathbf{1}_{n_l} & 0 & 0 \\
        0 & 1 & 0 \\
        0 & 0 & \mathbf{1}_{n_r} \\
        \end{pmatrix} \in \R^{k \times 3}
    \text{ and }
     G = \begin{pmatrix}
        T_{n_l} & * & 0 \\
        * & 0 & * \\
        0 & * & T_{n_r} \\
    \end{pmatrix} \in \R^{k \times k},
    \]
    where $*$ denotes any submatrices of the required dimensions. 
    Then it holds that 
    \begin{equation*}
    f(U,G) = 2^{k + 1 - |\{1,n_l\}| - |\{1,n_r\}|} \neq 0.
\end{equation*}
\end{lemma}


\begin{proposition}[Prop.~4.5 in \citealp{boege2024conditional}] \label{prop:trek}
    Let $\mathcal{T}: 1 \leftarrow \cdots \rightarrow p$ be a trek on $p$ nodes numbered from left to right, and let $X=(X_1,\dots,X_p)\sim\mathcal{N}(0,\Sigma)$ be an associated Gaussian random vector. Then there exists a choice of $\Sigma$ that satisfies the continuous Lyapunov equation for a stable matrix $M$ with zero pattern given by $\mathcal{T}$ such that 
    $
    {X_1 \centernot\ind X_p \mid X_2,\dots,X_{p-1}}$.
\end{proposition}

\begin{proof}
We may follow the approach from \Cref{ex:application} and construct a sequence of positive definite covariance matrices $\Sigma^{(m)} = \Sigma + \frac1m \Sigma^{(1,\infty)} + o\left(\frac1{m}\right)$ associated to the trek $\mathcal{T}$ such that \Cref{thm:main_convergence} is applicable, and it holds that $\Sigma_{1p\mid S} \neq 0$ for $S=\{2,\dots,p-1\}$.

As in \Cref{ex:application}, we construct a stable matrix $M^{(m)}$ by setting the diagonal entry $M^{(m)}_{j,j}$ to $-1$ if $j$ is a sink node, the source node, or adjacent to the source node of $\mathcal{T}$; otherwise, $M^{(m)}_{j,j}=-m$. The non-zero off-diagonal entries $M^{(m)}_{i,j}$ correspond to edges $j\to i$ of $\mathcal{T}$. We set $M^{(m)}_{i,j}=1$ if $j\to i$ is incident to the source or a sink of $\mathcal{T}$, and $M^{(m)}_{i,j}=m$, otherwise. The matrix $\Sigma^{(m)}$ is then the solution of the Lyapunov equation for $M^{(m)}$.

In the appendix, we give a display of the matrix $M^{(m)}$ and observe that $\Sigma^{(m)} = \Sigma + \frac1m \Sigma^{(1,\infty)} + o\left(\frac1{m}\right)$ for a positive semi-definite limit $\Sigma$ with $\Sigma_{1p\mid S} \neq 0$ and a matrix $\Sigma^{(1, \infty)}$ that contains Toeplitz structure as in \Cref{lem:det_GU}.  Conveniently, $\Sigma_{S,S} = U B U^T$ for a matrix $U$ as in \Cref{lem:det_GU}. \Cref{thm:main_convergence} becomes applicable and implies the claim.
\end{proof}

\section*{Acknowledgements}
The authors acknowledge support from the European Research Council (ERC) under the European Union’s Horizon 2020 research and innovation programme (grant agreement No 883818). Sarah Lumpp was further supported by the Munich Data Science Institute as well as the DAAD programme Konrad Zuse Schools of Excellence in Artificial Intelligence, sponsored by the Federal Ministry of Education and Research.


\bibliographystyle{plainnat}
\bibliography{bib}

\newpage

\appendix
\section{Proofs and additional calculations} 

\subsection{Proofs and calculations for Section 3} \label{sec:Appendix1}
In this section, we detail how in the case of a rank $r=1$ update, \cref{lem:matrix_det_generalized} defaults to the classic matrix determinant lemma. We further provide the proof of \cref{lem:matrix_det_taylor}.

\begin{remark} \label{rem:matrix_det_lemma}
In the special case where $r=1$ in \cref{lem:matrix_det_generalized}, we consider vectors $U = u \in \R^k$, $V = v \in \R^k$, and a scalar $B = b \in \R$. 
Applying the Cauchy Binet formula as in the proof of \cref{lem:matrix_det_generalized} to a rank 1 update of the not necessarily invertible matrix $\frac1m G$ yields a version of the matrix determinant lemma: 
\begin{equation*} 
\begin{aligned}
    \det(b uv^T + \frac1m G) 
    = \ & \sum_{i=1}^{k+1} 
    \det\left(\left[\begin{array}{c|c}
    \frac1m G & u \\
    \end{array}\right]_{[k],[k+1]\setminus \{i\}}\right) 
    \det\left(\left[\begin{array}{c|c}
    I_{k\times k} & bv \\
    \end{array}\right]_{[k],[k+1]\setminus\{i\}}\right) \\
    = \ & \sum_{i=1}^k  
    (-1)^{s([k-1]) + s([k])-i}
    \det\left(\left[\begin{array}{c|c}
        \left(\frac1m G\right)_{[k],[k]\setminus\{i\}} & u \\
    \end{array}\right]\right)
    b v_i + \det\left(\frac1m G\right) \\
    = \ & \frac{1}{m^{k-1}} \cdot b 
    \left(\sum_{i = 1}^{k} 
    (-1)^{k-i}v_i
    \sum_{j = 1}^k (-1)^{k + j} u_j \det\left(G_{[k]\setminus \{j\},[k]\setminus \{i\}}\right)\right) + \frac{1}{m^k} \det(G) \\
    = \ & \frac{1}{m^{k-1}} \cdot b 
    \left(\sum_{i,j = 1}^{k} 
    (-1)^{i+j}v_i
    \det\left(G_{[k]\setminus \{j\},[k]\setminus \{i\}}\right)u_j\right) 
    + \frac{1}{m^k} \det(G) \\
    = \ & \frac{1}{m^{k-1}} \cdot b v^T \mathrm{Adj}(G) u + \frac{1}{m^k} \det(G).
\end{aligned}
\end{equation*}
Note that in this case, the coefficient  $f_{\mathrm{asy}}(u,v,G) = v^T\mathrm{Adj}(G)u$ simplifies significantly. For the condition $f_{\mathrm{asy}}(u,v,G) \neq 0$ to be satisfied, $\mathrm{rank}(G) \geq n-1$ has to hold, as otherwise $\mathrm{Adj}(G) = 0$. If $\mathrm{rank}(G) = n-1$, we have $\mathrm{rank}(G)=1$, and if $G$ has full rank, $\mathrm{Adj}(G)$ has full rank as well. Then, the condition on $f_{\mathrm{asy}}$ translates to a non-orthogonality statement.
\end{remark}

\begin{proof}[Proof of \cref{lem:matrix_det_taylor}]
  Let $C^{(m)}=UBV^T+ \frac{1}{m}G$. By \cite{marcus1990determinants},
    \begin{multline}
        \det\left(C^{(m)} + D^{(m)}\right) 
        = \sum_{i = 0}^{k} \sum_{\substack{\gamma,\delta \subseteq [k] \\ |\gamma|=|\delta|=i}}(-1)^{s(\gamma)+s(\delta)}\det\left(C^{(m)}_{\gamma,\delta}\right)\det\left(D^{(m)}_{[k]\setminus \gamma,[k]\setminus \delta}\right) \\
        \label{eq:detC+D}
        = \det(C^{(m)}) + \sum_{i = 0}^{k-1} \sum_{\substack{\gamma,\delta \subseteq [k] \\ |\gamma|=|\delta|=i}}(-1)^{s(\gamma)+s(\delta)}\det\left(C^{(m)}_{\gamma,\delta}\right)\det\left(D^{(m)}_{[k]\setminus \gamma,[k]\setminus \delta}\right).
    \end{multline}
    For $\gamma,\delta \subseteq [k]$ with $|\gamma|=|\delta|=i$, we have 
    $
    \det\big(D^{(m)}_{[k]\setminus \gamma,[k]\setminus \delta}\big) = o\left(\frac{1}{m^{k-i}} \right).
    $
    \Cref{lem:matrix_det_generalized}  gives
    $
    \det\big(C^{(m)}_{\gamma,\delta}\big) = \frac{1}{m^{i-r}} c(\gamma,\delta) + o\left(\frac{1}{m^{i-r}}\right)
    $
    for some finite coefficient $c(\gamma,\delta) \in \R$. 
    Consequently,
    $
    \det\big(C^{(m)}_{\gamma,\delta}\big)
    \det\big(D^{(m)}_{[k]\setminus \gamma,[k]\setminus \delta}\big) 
                = 
        o\left(\frac{1}{m^{k-r}}\right),
    $
    which inserted into~\eqref{eq:detC+D} yields the claim of 
    \begin{equation*}
    \det(C^{(m)}+D^{(m)}) = \det(C^{(m)}) + o\left(\frac{1}{m^{k-r}}\right). \qedhere
    \end{equation*}
\end{proof}

\subsection{Proofs for Section 4}
Here, we provide a fully detailed proof of \cref{prop:trek}.  We begin with the computation of Toeplitz determinants that we require for the proof of \cref{lem:det_GU}.

\begin{lemma} \label{lem:toeplitz_det}
    Let $T_n$ and $T_n[i]$ be as in \cref{sec:toeplitz}. Then, 
    \begin{equation*}
        (a) \quad \det(T_n) = 2^{n-1}(1-n) \quad
        \text{ and } \quad
        (b) \quad 
        \det(T_n[i]) = 
        \begin{cases}
        2^{n-|\{1,n\}|}, & \text{ if } i \in \{1,n\}; \\
        0, & \text{ otherwise}.
        \end{cases}
    \end{equation*}
\end{lemma}

\begin{proof}
    Assume $n \geq 2$ as the case $n = 1$ is clear. The proof is based on straightforward Laplace expansion as well as identifying linearly dependent columns.
    
    $(a)$ Starting from the last row, we consecutively subtract the row above from the current row, and then perform the same procedure again as the lower $(n-2) \times n$ submatrix now contains only 1 and $-1$ as entries. The resulting matrix only contains one non-zero entry per row, allowing for simple Laplace expansion along these rows.

    $(b)$ For the second part of the statement, we observe that the first and last column of $T_n$ are the reverse of each other, implying that
    \[
    T_n(e_1 + e_n) = -(n-1) \mathbf{1}_n
    \]
    holds, where $e_j$ denotes the $j$-th unit vector in $\R^n$. Thus, the vector $\mathbf{1}_n$ is in the span of the first and last column of $T_n$. So in the case where $i \in \{2,\dots,n-1\}$, the first, $i$-th, and last column of $T_n[i]$ are linearly dependent, so the determinant is zero. 
    
    In the case that $i = 1$, we calculate the determinant with similar elimination steps as in $(a)$ yielding $\det(T_n[1]) = 2^{n-2}$.
    The case where $i = n$ can be traced back to the case with $i = 1$ by rotating the rows and columns accordingly.
\end{proof}

\begin{proof}[Proof of \cref{lem:det_GU}]
By definition,
\begin{align}
\label{eq:semmel}
        f\left(U,G\right)
        & = \sum_{\substack{\alpha \subseteq [k] \\ |\alpha|=k-3}} 
        (-1)^{s([k-3]) + s(\alpha)} 
        \det\left(\left[\begin{array}{c|c}
            G_{[k],\alpha} & U \\
        \end{array}\right]\right) 
        \det\left(U_{[k]\setminus \alpha,[3]}\right). 
    \end{align}
Let $i_t = n_l + 1$. Due to the block-diagonal structure of $U$, we have that $\det\left(U_{[k]\setminus \alpha,[3]}\right)\not=0$ if and only if $[k]\setminus \alpha =  \{i_l,i_t,i_r\}$  with $ i_l \in [n_l]$, $i_r \in [n_r] + i_t$.  It is easy to see that in this case, $\det\left(U_{[k]\setminus \alpha,[3]}\right)=1$.  For the other factor in the considered summands, we may calculate
    \begin{align*}
        \det\left(\left[G_{[k],\alpha} \mid U\right]\right) 
        &= \det 
        \left(\left[\begin{array}{ccc|ccc}
        \multicolumn{3}{c|}{\multirow{3}{*}{$
        \begin{pmatrix}
            T_{n_l} & * & 0 \\
            * & 0 & * \\
            0 & * & T_{n_r} \\
        \end{pmatrix}_{[k],[k]\setminus \{i_l,i_t,i_r\}}$}} & \mathbf{1}_{n_l} & 0 & 0 \\
            \multicolumn{3}{c|}{} & 0 & 1 & 0 \\
            \multicolumn{3}{c|}{} & 0 & 0 & \mathbf{1}_{n_r} \\
        \end{array}\right]\right) \\
        &=  
        (-1)^{(n_r-1 + n_l-i_l) + (n_r - 1) + (n_r - (i_r-i_t))}
        \det\left(\begin{array}{ccc}
            T_{n_l}[i_l] & 0 & 0 \\
            * & 1 & * \\
            0 & 0 & T_{n_r}[i_r-i_t] \\
        \end{array}\right) \\
        &= (-1)^{k - (i_l + (i_r - i_t) + 1)} \det\left(T_{n_l}[i_l]\right) \det\left(T_{n_r}[i_r-i_t]\right),  
    \end{align*}
    which is non-zero if and only if $\det\left(T_{n_l}[i_l]\right)\not=0$ and $\det\left(T_{n_r}[i_r-i_t]\right)\not=0$. This is the case if and only if $i_l \in \{1,n_l\}$ and $i_r-i_t \in \{1,n_r\}$ as shown in \cref{lem:toeplitz_det}, yielding
    \begin{align}
    \label{eq:leberkaese}
        \det\left(\left[G_{[k],\alpha} \mid U\right]\right) 
        &= (-1)^{k - (i_l + (i_r - i_t) + 1)}2^{k- 1 - |\{1,n_l\}| - |\{1,n_r\}|}. 
    \end{align}
 Inserting \eqref{eq:leberkaese} into formula~\eqref{eq:semmel} and simplifying the signs, we arrive at
\begin{equation*}
    f(U,G) = 2^{k + 1 - |\{1,n_l\}| - |\{1,n_r\}|} \neq 0.
    \qedhere
\end{equation*}
\end{proof}

\begin{remark}
    With similar arguments as in the proof of \cref{lem:det_GU}, the statement can be extended to the case $n_l = 0$, that is, $U = \begin{pmatrix}
        1 & 0 \\ 0 & \mathbf{1}_{n_r}
    \end{pmatrix}$ and $G = \begin{pmatrix}
        0 & * \\ * & T_{n_r}
    \end{pmatrix}$ or even $U = \mathbf{1}_{n_r}$ and $G = T_{n_r}$.
\end{remark}

\begin{proof}[Proof of \cref{prop:trek}]
    Let $S = \{2,\dots,p-1\}$. Let $a_l$ and $a_r$ be the number of edges 
    on the left and the right branch of the trek $\mathcal{T}$, respectively.  Then $p = a_l + 1 + a_r$, and the node numbered $a_l+1$ is the source of $\mathcal{T}:1\leftarrow 2\leftarrow \dots\leftarrow  a_l\leftarrow a_l+1\to \dots \to p$. 
    \medskip
    
    \textbf{Case 1.} First, we consider the case where $a_l \geq 2$ and $a_r > 2$. Note that the case $a_l = 2$ and $a_r = 3$ is considered in \cref{ex:application}. The general proof follows the same strategy as described in the example.
    Consider the sequence of stable matrices 
    \[
    M^{(m)} = \left(\begin{array}{c | c c c c c | c | c c c c c | c }
        -1 & 1 & 0 & \multicolumn{2}{c}{\cdots} & 0 & & & & & & & \\
        \hline
        & -m & m & & & 0 & & & & & & & \\
        & & \ddots & \ddots & & \vdots & & & & & & \\
        & & & -m & m & & & & & & & \\
        & &  &  & -m & m & & & & & & \\
        & & & & & -1 & 1 & & & & & & \\
        \hline
        \rule{0pt}{\normalbaselineskip}
        & & & & & & -1 & & & & & & \\[1.5mm]
        \hline
        & & & & & & 1 & -1 & & & & & \\
        & & & & & & & m & -m & & & & \\
        & & & & & & & 0 & m & -m & & & \\
        & & & & & & & \vdots & & \ddots & \ddots & & \\
        & & & & & & & & & & m & -m & \\
        \hline
        & & & & & & & 0 & \multicolumn{2}{c}{\cdots} & 0 & 1 & -1 
        \end{array}\right) \in \R^{p \times p}
    \]
    where all entries not indicated are zero. The four upper left blocks form a submatrix of dimension $a_l \times a_l$, while the four lower right blocks form a submatrix of dimension $a_r \times a_r$. In the case that $a_l = 2$, the four upper left blocks form the submatrix 
    $\begin{bmatrix}
        -1 & 1 \\ 0 & -1
    \end{bmatrix}$. 
    Note that the matrix $M^{(m)}$ encodes the interactions of the coordinates of $X$ by satisfying $M^{(m)}_{i,j} = 0$ if the edge $j \rightarrow i$ is not contained in the graph $\mathcal{T}$.
    
    Solving the Lyapunov equation \eqref{eq:lyap} defined by these matrices $M^{(m)}$ for every $m \in \N$ yields a sequence of positive definite covariance matrices $\Sigma^{(m)}$ satisfying
    \begin{equation} \label{eq:expansion}
         \Sigma^{(m)} = \Sigma + \Sigma^{(1,\infty)}\frac1m + \mathcal{O}\left(\frac{1}{m^2}\right).
    \end{equation}
    In particular, the remainder is of order $\mathcal{O}\left(\frac{1}{m^2}\right) \subseteq o\left(\frac1{m}\right)$. The matrices $\Sigma$ and $\Sigma^{(1,\infty)}$ can be computed with a similar ansatz as in the proof of Lemma~4.7 of \citet{boege2024conditional} yielding
    \[\renewcommand*{\arraystretch}{1.5}
    \Sigma = 
    \left(\begin{array}{c | c c c | c }
        \frac{15}8 & \frac78 \mathbf{1}^T & \frac14 & \frac38 \mathbf{1}^T & \frac38 \\
        \hline
        \frac78 \mathbf{1} & \frac32 \mathbf{1}\mathbf{1}^T & \frac12 \mathbf{1} & \frac12 \mathbf{1}\mathbf{1}^T & \frac38 \mathbf{1} \\
        \frac14 & \frac12 \mathbf{1}^T & 1 & \frac12 \mathbf{1}^T & \frac14 \\
        \frac38 \mathbf{1} & \frac12 \mathbf{1}\mathbf{1}^T & \frac12 \mathbf{1} & \frac32 \mathbf{1}\mathbf{1}^T & \frac78 \mathbf{1}\\
        \hline
        \frac38 & \frac38 \mathbf{1}^T & \frac14 & \frac78 \mathbf{1}^T & \frac{15}8 \\
    \end{array}\right)
    \text{ and }
    \renewcommand*{\arraystretch}{1.5}
    \Sigma^{(1,\infty)} = 
    \left(\begin{array}{c | c c c | c }
        * & \multicolumn{3}{c|}{*} & * \\
        \hline
        \multirow{3}{*}{$*$} & T_{n_l} & * & 0 & \multirow{3}{*}{$*$}\\
        & * & 0 & * & \\
        & 0 & * & T_{n_r} & \\
        \hline
        * & \multicolumn{3}{c|}{*} & * \\
    \end{array}\right) \in \R^{p \times p},
    \]
    where $*$ encodes suitably sized blocks of arbitrary values and $n_l = a_l - 1$, $n_r = a_r - 1$.
    Note that $\Sigma = \lim_{m \to \infty}\Sigma^{(m)}$ is positive semi-definite and singular due to the duplicated rows and columns.
    Computing the entry of the Schur complement corresponding to the conditional independence statement, we obtain $\Sigma_{1p\mid S} = \frac38 - \frac{13}{48} = \frac{5}{48} \neq 0$, that is, the conditional independence does not hold in the distribution induced by this singular covariance matrix. 
    
    To prove the proposition, however, we require a distribution with positive definite covariance matrix satisfying \eqref{eq:lyap} and the conditional independence statement in \cref{prop:trek}.
    Even though the Schur complement does not converge in general, \cref{thm:main_convergence} now allows us to show convergence of the Schur complement along this specific sequence and thus resort to a continuity argument to find a positive definite covariance matrix along this sequence where the considered entry of the Schur complement is non-zero as well.
    
    It is clear that $\Sigma = \hat{U}\hat{B}\hat{U}^T$ with
    \[\renewcommand*{\arraystretch}{1.5}
    \hat{U} = \left(\begin{array}{c | ccc | c }
        1 & 0 & 0 & 0 & 0  \\
        \hline
        0 & \mathbf{1}_{n_l} & 0 & 0 & 0 \\
        0 & 0 & 1 & 0 & 0 \\
        0 & 0 & 0 & \mathbf{1}_{n_r} & 0 \\
        \hline
        0 & 0 & 0 & 0 & 1 \\
    \end{array}\right) \in \R^{p\times 5}
    \text{ and }
    \hat{B} = \left(\begin{array}{c | c c c | c }
        \frac{15}8 & \frac78  & \frac14 & \frac38 & \frac38 \\
        \hline
        \frac78 & \frac32 & \frac12 & \frac12 & \frac38 \\
        \frac14 & \frac12 & 1 & \frac12 & \frac14 \\
        \frac38 & \frac12 & \frac12 & \frac32 & \frac78 \\
        \hline
        \frac38 & \frac38 & \frac14 & \frac78 & \frac{15}8 \\
    \end{array}\right) \in \R^{5 \times 5}.
    \]
    Observe that $\hat{U}$ has full column rank and $\hat{B}$ is invertible. Further note that $\Sigma^{(1,\infty)}_{S,S}$ coincides with the matrix $G$ in \cref{lem:det_GU}.
    Now, we can write
    \[
    \Sigma_{S,S} = \left(\hat{U}\hat{B}\hat{U}^T\right)_{S,S} = \hat{U}_{S,[5]}\hat{B}(\hat{U}^T)_{[5],S} = \hat{U}_{S,234}\hat{B}_{234,234}(\hat{U}^T)_{234,S}
    \]
    since the first and last column of $\hat{U}_{S,[5]}$ are zero columns. These factorization together with the Toeplitz structure of $\Sigma^{(1,\infty)}_{S,S}$ allow us to verify conditions $(i)$ and $(ii)$ in \cref{thm:main_convergence}.
    We use these results to verify the conditions of \Cref{thm:main_convergence} as follows. 
    \begin{itemize}
        \item[(i)] We have $\Sigma_{S,S} = U B U^T$ with $U = \hat{U}_{S,234}$ and $B = \hat{B}_{2 3 4,2 3 4}$, where $\mathrm{rank}\left(\Sigma_{S,S}\right) = 3 < k$, $U$ has full column rank, and $B$ is invertible. Applying \Cref{lem:det_GU} yields $f\left(U,\Sigma^{(1,\infty)}_{S,S}\right) \neq 0$.
        \item[(ii)] Due to the structure of $U$, we have 
        \[
        \left(\Sigma_{1,S}\right)^T 
        = \left(\hat{U} \hat{B} \hat{U}^T\right)_{S,1} 
        = \left(\hat{U} \hat{B}\right)_{S,1} 
        = \hat{U}_{S,[5]} \hat{B}_{[5],1}
        = U B_{[3],1} \in \mathrm{im}(U) = \mathrm{im}(\Sigma_{S,S}),
        \]
        and with a similar computation $\Sigma_{S,p} \in \mathrm{im}(\Sigma_{S,S})$.
    \end{itemize}
    
    \Cref{thm:main_convergence} then yields that 
    $
    \lim_{m \to \infty} \Sigma^{(m)}_{1p\mid S} = \Sigma_{1p\mid S} \neq 0;
    $
    thus, there exists a distribution with positive definite covariance matrix solving \eqref{eq:lyap} where
    $
    X_1 \centernot\ind X_p \mid X_2,\dots,X_{p-1}
    $
    holds.
    \medskip
    
    \textbf{Case 2.} The previously excluded cases can be handled by adapting the structure of $M^{(m)}$ accordingly. 
    Assume without loss $a_l \in \{0,1\}$ and $a_r \geq 2$. Thus, we have $p = 2 + a_r \geq 4$ for $a_l = 1$ and $p = 1 + a_r \geq 3$ for $a_l = 0$. 
    In the case $a_l = 1$, consider the sequence of stable matrices
    \[
    M^{(m)} = \left(\begin{array}{c | c | c c c c c | c }
        -1 & 1 &  & & & & & \\
        \hline
        \rule{0pt}{\normalbaselineskip}
        & -1 & 0 & & & & & \\[1.5mm]
        \hline
        \rule{0pt}{\normalbaselineskip}
        & 1 & -1 & & & & & \\
        & & m & -m & & & & \\
        & & 0 & m & -m & & & \\
        & & \vdots & & \ddots & \ddots & & \\
        & & 0 & & & m & -m & \\
        \hline
        & & & 0 & \cdots & 0 & 1 & -1 
        \end{array}\right) \in \R^{p \times p},
    \]
    which induces a sequence of positive definite solutions $\Sigma^{(m)}$ to the Lyapunov equation that can be expanded as in \eqref{eq:expansion} with 
    \[
    \renewcommand*{\arraystretch}{1.5}
    \Sigma = 
    \left(\begin{array}{c | c c | c }
        \frac32 & \frac12 & \frac12 \mathbf{1}^T & \frac38 \\
        \hline
        \frac12 & 1 & \frac12 \mathbf{1}^T & \frac14 \\
        \frac12 \mathbf{1} & \frac12 \mathbf{1} & \frac32 \mathbf{1}\mathbf{1}^T & \frac78 \mathbf{1}\\
        \hline
        \frac38 & \frac14 & \frac78 \mathbf{1}^T & \frac{15}8 \\
    \end{array}\right)
    \text{ and }
    \renewcommand*{\arraystretch}{1.5}
    \Sigma^{(1,\infty)} = 
    \left(\begin{array}{c | c c | c }
        * & \multicolumn{2}{c|}{*} & * \\
        \hline
        \multirow{2}{*}{$*$} & 0 & * & \multirow{2}{*}{$*$} \\
        & * & T_{n_r} & \\
        \hline
        * & \multicolumn{2}{c|}{*} & * \\
    \end{array}\right) \in \R^{p \times p}.
    \]
    Here, we have $\Sigma_{1p\mid S} = \frac{1}{10} \neq 0$, and can write $\Sigma = \hat{U}\hat{B}\hat{U}^T$ with
    \[\renewcommand*{\arraystretch}{1.5}
    \hat{U} = \left(\begin{array}{c | cc | c }
        1 & 0 & 0 & 0  \\
        \hline
        0 & 1 & 0 & 0 \\
        0 & 0 & \mathbf{1}_{n_r} & 0 \\
        \hline
        0 & 0 & 0 & 1 \\
    \end{array}\right) \in \R^{p\times 4}
    \text{ and }
    \hat{B} = \left(\begin{array}{c | c c | c }
        \frac32 & \frac12 & \frac12 & \frac38 \\
        \hline
        \frac12 & 1 & \frac12 & \frac14 \\
        \frac12 & \frac12 & \frac32 & \frac78 \\
        \hline
        \frac38 & \frac14 & \frac78 & \frac{15}8 \\
    \end{array}\right) \in \R^{4 \times 4}.
    \]
    Again, observe that $\hat{U}$ has full column rank and $\hat{B}$ is invertible while $\Sigma^{(1,\infty)}_{S,S}$ coincides with the matrix $G$ in \cref{lem:det_GU} when $n_l = 0$.
    We can further write
    \[
    \Sigma_{S,S} = \left(\hat{U}\hat{B}\hat{U}^T\right)_{S,S} = \hat{U}_{S,[4]}\hat{B}(\hat{U}^T)_{[4],S} = \hat{U}_{S,23}\hat{B}_{23,23}(\hat{U}^T)_{23,S}
    \]
    since the first and last column of $\hat{U}_{S,[4]}$ are again zero columns. Conditions $(i)$ and $(ii)$ in \cref{thm:main_convergence} can again be verified as described in Case 1.
    
    The case $a_l = 0$ corresponds to a directed path as the underlying graph $\mathcal{T}$. To define $M^{(m)}$ in this case, we remove the first row and column of the corresponding matrix in the case $a_l = 1$. The matrices obtained from the resulting sequence of covariance matrices are then exactly the $2\dots p,2\dots p$ submatrices of the corresponding matrices in the case $a_l  =1$. 
    In particular, we have $\Sigma_{1p\mid S} = -\frac1{24} \neq 0$, and
    \[\renewcommand*{\arraystretch}{1.5}
    \hat{U} = \left(\begin{array}{c | c | c }
        1 & 0 & 0 \\
        \hline
        0 & \mathbf{1}_{n_r} & 0 \\
        \hline
        0 & 0 & 1 \\
    \end{array}\right) \in \R^{p\times 3}
    \text{ and }
    \hat{B} = \left(\begin{array}{c | c | c }
        1 & \frac12 & \frac14 \\
        \hline
        \frac12 & \frac32 & \frac78 \\
        \hline
        \frac14 & \frac78 & \frac{15}8 \\
    \end{array}\right) \in \R^{3 \times 3},
    \]
    yielding similar conclusions as in the case $a_l = 1$.
    Note that in the case $a_l = 0$, when verifying the conditions of \cref{thm:main_convergence}, the considered matrix $U = \mathbf{1}_{n_r}$ is a vector, while the considered matrix $G = T_{n_r}$. Thus, to prove \cref{thm:main_convergence} in the case of a distribution on a directed path satisfying the Lyapunov equation, it suffices to employ the classic matrix determinant lemma for rank 1 updates in \cref{lem:matrix_det_generalized} and, subsequently, in \eqref{eq:denominator} and \eqref{eq:numerator}.
    \medskip

    \textbf{Case 3.} We can generally assume without loss that $a_r \geq a_l$, as otherwise the numbering of the nodes can be reversed. The remaining cases are then the following: $a_l = 2$ and $a_r = 2$, $a_l = 1$ and $a_r = 2$, $a_l = 1$ and $a_r = 1$, as well as $a_l = 0$ and $a_r = 2$, and the bivariate case $a_l = 0$ and $a_r = 1$.  For these cases, we can construct explicit positive definite covariance matrices $\Sigma$ that solve \eqref{eq:lyap} and satisfy the conditional independence statement in \cref{prop:trek}, for instance by setting the diagonal of the stable matrix $M$ to~$-1$ and all potentially non-zero off-diagonal entries to~$1$. 
\end{proof}

\vspace{2cm}

\end{document}